\documentclass[english,a4paper,reqno]{amsart}

\usepackage{amsfonts}
\usepackage{mathtools}

\usepackage{amssymb}
\usepackage{amsmath}
\usepackage{bbm}
\usepackage{bm}
\usepackage{dsfont}
\usepackage{graphicx}

\usepackage{enumitem}
\usepackage{tikz}
\usepackage{subfigure}
\usepackage{ifpdf}

\newtheorem{theorem}{Theorem}
\newtheorem{proposition}{Proposition}

\newtheorem{remark}{Remark}
\newtheorem{corollary}{Corollary}

\newcommand{\matc}{\ensuremath{\mathcal{C}}}

\newcommand{\mreals}{\ensuremath{\mathbb{R}}}

\newcommand{\FF}{\ensuremath{\mathbb{F}}}

\ifx\eqref\undefined
	\newcommand{\eqref}[1]{~(\ref{#1})}
\fi
\ifx\mod\undefined
	\def\mod{\mathop{\rm mod}}
\fi

\def\EE{\mathbb{E}\,}

\def\eqdef{\stackrel{\triangle}{=}}

\ifx\lesssim\undefined
\def\simleq{{{\mskip 3mu plus 2mu minus 1mu%
	\setbox0=\hbox{$\mathchar"013C$}%
	\raise.2ex\copy0\kern-\wd0%
	\lower0.9ex\hbox{$\mathchar"0218$}}\mskip 3mu plus 2mu minus 1mu}}
\else
\def\simleq{\lesssim}
\fi
\ifx\gtrsim\undefined
\def\simgeq{{{\mskip 3mu plus 2mu minus 1mu%
	\setbox0=\hbox{$\mathchar"013E$}%
	\raise.2ex\copy0\kern-\wd0%
	\lower0.9ex\hbox{$\mathchar"0218$}}\mskip 3mu plus 2mu minus 1mu}}
\else
\def\simgeq{\gtrsim}
\fi

\begin{document}

\sloppy

\title{Bounds for codes on pentagon and other cycles} 

\author{Marco~Dalai and Yury Polyanskiy 
}
\thanks{M. D. is with the Department of Information Engineering, University of Brescia, Italy, e-mail:
\texttt{marco.dalai@unibs.it}. Y.P. is with the Department of EECS, MIT, Cambridge, MA, email: \texttt{yp@mit.edu}. 
This material is based upon work supported by the National Science
    Foundation under Grant No CCF-13-18620. The work was mainly completed while the authors were visiting the Simons Institute for the
    Theory of Computing at UC Berkeley, whose support is gratefully acknowledged.
}%

\begin{abstract}
The capacity of a graph is defined as the rate of exponential grow of independent sets in the strong powers of the
graph. In strong power, an edge connects two sequences if at each position letters are equal or adjacent. We consider a
variation of the problem where edges in the power graphs are removed among sequences which differ in more than a
fraction $\delta$ of coordinates.  For odd cycles, we derive an upper bound on the corresponding rate which
combines Lov\'asz' bound on the capacity with Delsarte's linear programming bounds on the minimum distance of codes in
Hamming spaces.  For the pentagon, this shows that for $\delta \ge {1-{1\over\sqrt{5}}}$ the Lov\'asz rate is the best
possible,  while we prove by a Gilbert-Varshamov-type bound that a higher rate is achievable for $\delta < {2\over 5}$.

Communication interpretation of this question is the problem of sending quinary symbols
subject to $\pm 1\mod 5$ disturbance. The maximal communication rate subject to the zero undetected-error equals
capacity of a pentagon. The question addressed here is how much this rate can be increased if only a fraction $\delta$
of symbols is allowed to be disturbed.  
\end{abstract}

\maketitle

\section{Introduction}

The problem we consider in this paper is the following. Given a graph $G$ we define a semimetric on the vertex set $V(G)$ 
$$ d(v,v') = \begin{cases} 0, & v=v',\\
			1, &v\sim v',\\
			\infty, & \mbox{otherwise} \end{cases}\,.$$
We extend this semimetric additively to the Cartesian products $V(G)^n$ and define a graph $G(n,d)$ as follows
$$ V(G(n,d)) = V(G)^n, E(G(n,d))=\left\{(x,x'): d(x,x') \eqdef \sum_{j=1}^n d(x_j, x_j') \le d \right\}\,.$$
The goal is to determine (bounds on)
$$ R^*(G, \delta) \eqdef \limsup_{n\to\infty} {1\over n } \log \alpha(G(n,\delta n))\,.$$

Note that $G(n,d)$ corresponds to the graph obtained by removing in the strong power graph $G^n$ edges between sequences which differ in more than $d$ positions. On one hand, this problem is a specialization of the general one considered in \cite{dalai-TIT-2015} (see Section V.D therein). On
the other hand, it is a natural generalization of the two classically studied ones:
\begin{enumerate}
\item Shannon capacity of a graph~\cite{shannon-1956}, which corresponds to $\delta=1$. The best general upper bound is~\cite{lovasz-1979}
	\begin{equation}\label{eq:lovasz}
	 R^*(G, 1) \le \log \theta_L(G)\,,
\end{equation}	
	where $\theta_L$ is the Lovasz's $\theta$-function (see below).
\item Rate-Distance tradeoff in Hamming spaces corresponds to $G=K_q$ (the clique). Here the two bounds we mention are 
\begin{equation}\label{eq:hamming}
	R_{GV}(q,\delta) \le R^*(K_q,\delta) \le R_{LP1}(q,\delta)\,,
\end{equation}
where for $\delta < 1-{1\over q}$
\begin{align} R_{GV}(q,\delta) &\eqdef \log q - H_q(\delta)\\
   R_{LP1}(q,\delta) &\eqdef H_q\left({(q-1)-(q-2)\delta - 2\sqrt{(q-1)\delta(1-\delta)}\over q}\right)\,,\\
   H_q(x) &\eqdef x \log(q-1) - x \log x - (1-x)\log(1-x)\,,
\end{align}   
and for $\delta\ge 1-{1\over q}$ both $R_{GV}$ and $R_{LP1}$ equal zero.
(Better bounds also exist: an improved upper bound for for small $\delta$'s was found by Aaltonen~\cite{aaltonen1990new}, and an
improved lower bound for large $q$'s and some range of $\delta$'s is shown via algebraic-geometric codes~\cite{tsfasman1982modular}.)
\end{enumerate}

The main contribution of this work is an upper bound on $R^*(G,\delta)$ for the case of $G=C_q$ (a $q$-cycle,
$q\ge3$--odd). Namely,
we prove:
\begin{equation}\label{eq:main}
	R^*(C_q, \delta) \le \log \theta_L(C_q) + R_{LP1}(q', \delta)\,, \qquad q' = 1+{1\over \cos{\pi/q}} \le q\,. 
\end{equation}

Note that, for $q>3$, $q'$ is not an integer, and $R_{LP1}(q',\delta)$ can no longer be thought of as a rate upper bound for Hamming
space.  Also note that $\theta_L(C_q) = {q\over q'}$ and so (informally speaking) the upper bound decomposes $C_q^n$
into disjoint $\left({q\over q'}\right)^n$ copies of imaginary Hamming spaces with fractional $q'$, inside which one can
additionally pack exponentially many points at minimum distance $\delta$.
As in the ordinary linear programming bound, the derivation of the bound goes through the use of  Krawtchouk polynomials; to the best of our knowledge, this is the first appearance of these polynomials with a non-integer coefficient in
coding theory or combinatorics.

We also give some lower bounds; in particular, for the pentagon we prove 
\begin{equation}\label{eq:main_penta}
	{1\over 2} \log 5 + {1\over 2} R_{GV}(5, 2\delta) \le R^*(C_5, \delta) \le {1\over 2} \log 5 + R_{LP1}(\sqrt{5}, \delta)\,.
\end{equation}

\section{Minimum Distance of Codes on Cycles}

In the rest of the paper, we will focus on cycles. Let for ease of notation  $C_q = \mathbb{Z}/q\mathbb{Z}$, be a ring and a graph  (the $q$-cycle) simultaneously.
A code $\matc$ is a subset of $C_q^n$ and its minimum distance is defined as usual
\begin{equation}
d_{\min}(\matc)=\min_{c\neq c' \in\matc} d(c,c').
\end{equation}
Let $M_q(n,d)$ be the size of the largest code of length $n$ with minimum distance at least $d$, i.e. set $M_q(n,d)=\alpha(C_q(n,d-1))$.

\subsection{Even Cycles}
For even cycles, the problem is equivalent to the ordinary binary case with Hamming distance.
\begin{proposition}
\label{prop:evenq}
For even $q$, we have
\begin{equation}
M_q(n,d)=(q/2)^n M_{2}(n,d).
\end{equation}
and hence
\begin{equation}
R^*(C_q,\delta)=\log(q/2)+R^*(K_2,\delta).
\end{equation}
\end{proposition}
\begin{proof}
For any $v\in\mathbb{Z}_{q/2}^n$, consider the sequence of edges indexed by $v$
\begin{equation}
\mathcal{E}_{v}=\bigtimes_{i=1}^n \{2v_i,2v_i+1\}.
\end{equation}
There are $(q/2)^n$ such sets, and each codeword is contained in one of them. Hence, for a given $q$-ary code $\matc$ which achieves $M_q(n,d)$, one of those sets contains at least $\lceil M_q(n,d)(q/2)^{-n}\rceil$ codewords. Since these codewords are all at finite distance with one another, they must differ in at least $d$ positions, and they can be mapped to binary codewords preserving pairwise distances. Hence, $\lceil M_q(n,d)(q/2)^{-n}\rceil\leq M_2(n,d)$, implying $M_q(n,d)\leq (q/2)^n M_{2}(n,d)$.

Conversely, let $\matc'\subseteq \{0,1\}^n$ be any binary code achieving $M_2(n,d)$.
Let $\matc=\{0,2,\ldots,q-2\}^n$ and define the code
\begin{equation}
\matc^+ = \matc +\matc'
\end{equation}
Consider two distinct codewords $x^+=x+x'$ and $y^+=y+y'$ in $\matc^+$. If $x'=y'$ then $x\neq y$ and $d(x^+,y^+)=\infty$. Otherwise, $x^+-y^+$ contains at least $d$ odd components and thus $d(x^+,y^+)\geq d$. Since $|\matc|=(q/2)^n M_{2}(n,d)$, we deduce that $M_q(n,d)\geq (q/2)^n M_{2}(n,d)$. Note that the construction of $\matc^+$ is a particular case of the general one discussed in Section \ref{sec:achievability} below.
\end{proof}

\subsection{Odd cycles}
Proposition \ref{prop:evenq} relies on the fact that, for even $n$, the clique covering number of $C_q^n$ equals its independence number. For odd cycles, a straightforward attempt along the same line only gives a much weaker result.
\begin{proposition}
\label{prop:oddqweak}
For odd $q$, we have
\begin{equation}
\log((q-1)/2)+R^*(K_2,\delta)\leq R^*(C_q,\delta)\leq\log(q/2)+R^*(K_2,\delta).
\end{equation}
\end{proposition}
\begin{proof}
The first inequality is trivial, since any $\delta$-code for $C_{q-1}^n$ is a $\delta$-code for $C_{q}^n$. To prove the upper bound, 
let $\bar{\chi}(C_q^n)$ be the clique covering number of $C_q^n$. Any code achieving $M_q(n,d)$ has at least $\lceil M_q(n,d)/\bar{\chi}(C_q^n)\rceil $ codewords in the same clique and these codewords all lie in the same edge in each position. So, they can be mapped to a binary code and thus $\lceil M_q(n,d)/\bar{\chi}(C_q^n)\rceil \leq M_2(n,d)$. As $n\to \infty$,  $\bar{\chi}(C_q^n)^{1/n}$ tends to $\bar{\chi}^*(C_q)=q/2$, the fractional clique covering number of the cycle. Taking logarithms and dividing by $n$ we deduce $R_q(\delta)\leq\log(q/2)+R_{2}(\delta)$.
\end{proof}
\begin{remark}
Proposition \ref{prop:oddqweak} implies that, for large $q$, 
\begin{equation}
R^*(C_q,\delta)=\log(q/2)+R^*(K_2,\delta)+\mathcal{O}(1/q)\,.
\end{equation} 
\end{remark}

Proposition \ref{prop:oddqweak} gives poor results for large values of $\delta$. In particular, the upper bound 
becomes weaker than Lov\'asz' result
\begin{equation}\label{eq:lovasz_cq}
	R^*(C_q,1)\leq \log \theta_L(C_q), \qquad \theta_L(C_q) = \frac{q \cos(\pi/q)}{1+\cos(\pi/q)}.
\end{equation}
The first question we address is the following: Is it possible to show $R^*(C_q,\delta) \le \log \theta_L(C_q)$ for
some $\delta<1$? That is, can we give a tighter bound on the ``Plotkin'' point?
The answer is yes and we can do so without the (heavier) machinery needed for~\eqref{eq:main}.

Our upper bounds require computing $\theta$-functions. We define three variations of those as
follows (see \cite{schrijver-1979}, \cite{lovasz-1979} and \cite{szegedy-1994})
\begin{align} \theta_S(G) &\eqdef  \min\{\max_{v\in V(G)} D_{v,v}: D\succeq J, D|_{E(\bar G)}\le 0\}\label{eq:thetas_alt}\\
	      \theta_L(G) &\eqdef  \min\{\max_{v\in V(G)} D_{v,v}: D\succeq J, D|_{E(\bar G)}= 0\}\label{eq:thetal_alt}\\
		\theta_z(G) &\eqdef \min\{\max_{v\in V(G)} D_{v,v}: D\succeq J, D|_{E(\bar G)}= 0, D_{v,v'} \ge 0\}\,,
					\label{eq:thetaz_alt}
\end{align}
where $D$ is $|G|\times|G|$ matrix and $J$ is a matrix of all ones.
These stand for the Schrijver, Lov\'asz and Szegedy's $\theta$-functions respectively. 


\begin{proposition}\label{prop:simple} We have
	\begin{equation}\label{eq:prop_simple}
		\theta_S(G(n,d)) \le \theta_z(G)^n \theta_S(K_{|G|}, d)\,.
\end{equation}	
\end{proposition}
\begin{proof} We prove a more general result: Let $G$ and $H$ be graphs with the same vertex sets
$V(G)=V(H)$, then 
\begin{equation}\label{eq:ths_intersection}
	\theta_S(G\cap H) \le \theta_z(G) \theta_S(H)\,,
\end{equation}
where $G\cap H$ denotes a graph obtained by intersecting edge-sets $E(G)\cap E(H)$. Let $D_1$ and $D_2$ be solutions
in~\eqref{eq:thetaz_alt} and~\eqref{eq:thetas_alt} for the graphs $G$ and $H$, respectively. Then, 
$$ D_{v,v'} \eqdef (D_1)_{v,v'} (D_2)_{v,v'}\,, $$
i.e. $D=D_1 \odot D_2$, 
is a feasible choice for $\theta_S(G\cap H)$. Indeed, by property of Schur product of matrices
$$ (D_1 - J)\odot(D_2-J) \succeq 0\,,$$
and expanding the left-hand side we get $D\succeq J$ as required. The condition $D_{E(\overline{G\cap H})}$ is satisfied
trivially.

To complete the proof of the proposition one only needs to notice that by taking tensor-power of a solution $D$ for
$\theta_z(G)$ we get $\theta_z(G^n)\le \theta_z(G)^n$.

Finally, to derive~\eqref{eq:prop_simple} from~\eqref{eq:ths_intersection} we note that 
$$G(n,d) =  G^n \cap K_{|G|}(n, \delta n)\,,$$
where $G^n$ denotes the strong-product power. 
\end{proof}
Notice that for cycles
$$ \theta_z(C_q)=\theta_L(C_q)\,.$$
Therefore, Proposition \ref{prop:simple} implies the following.
\begin{corollary}\label{cor:sz+hamming}
$$ R^*(C_q, \delta n) \le \log \theta_L(C_q) + R_{LP1}(q, \delta)\,.$$
\end{corollary}

This shows that $R^*(C_q,\delta) \le \log \theta_L(C_q)$ for $\delta\geq 1-1/q$, improving \eqref{eq:lovasz_cq}.
 While Proposition \ref{prop:oddqweak} is weak for large $\delta$, Corollary \ref{cor:sz+hamming} becomes quickly rather weak as $\delta$ decreases from $(q-1)/q$. 
We next proceed to our main result, which improves Corollary \ref{cor:sz+hamming} and gives a bound uniformly good at all values of $\delta$. 
\begin{theorem}\label{th:main}
\begin{equation}\label{eq:main2}
	R^*(C_q, \delta) \le \log \theta_L(C_q) + R_{LP1}(q', \delta)\,, \qquad q' = 1+{1\over \cos{\pi/q}} \le q\,. 
\end{equation}
\end{theorem}
\begin{proof}
The bound is based on $\theta$ functions and Delsarte's linear programming bound \cite{delsarte-1973}, but it is easier to describe in terms of Fourier transforms. To that end we
identify $C_q$ with the ring $\mathbb{Z}/q\mathbb{Z}$. 

For any 
$f:C_q^n\to \mathbb{C}$ we define its Fourier transform as
\[
\hat f(\omega) = \sum_{x\in C_q^n} f(x)e^{\frac{2\pi i}{q}<\omega,x>},\quad \omega\in C_q^n
\]
where the non-degenerate $C_q$-valued bilinear form is defined as usual
$$<x,y>\eqdef \sum_{i=1}^n x_i y_i\,. $$ 
We also define the inner product as follows
$$ (f,g) \eqdef q^{-n} \sum_{x\in C_q^n} \bar f(x) g(x)\,. $$

The starting point is a known rephrasing of linear programming bounds. Let $\matc$ be a code with minimum distance at
least $d$. Let $f$ be such that $f(x)\leq 0$ if $d(x,0)\geq d$, $\hat{f}\geq 0$ and $\hat{f}(0)>0$. Then,  consider the Plancherel  identity
\begin{equation}
(f* 1_\matc, 1_\matc)=q^{-n}(\hat{f}\cdot\widehat{1_\matc}, \widehat{1_\matc})\,,
\end{equation}
where $1_A$ is the indicator function of a set $A$. Upper bounding the left hand side by $|\matc|f(0)$ and lower
bounding the right hand side by the zero-frequency term $q^{-n}\hat{f}(0)|\matc|^2$, one gets
\begin{align}\label{eq:f0fhat0}
|\matc| & \leq \min \left\{q^n \frac{f(0)}{\hat{f}(0)}: f(x) \le 0\mbox{~for~} d(x,0)\geq d, \hat f \ge 0, \hat{f}(0)>0 \right\}\,.
\end{align}
The proof of our theorem is based on an assignment $f$ which combines Lov\'asz' choice used to obtain \eqref{eq:lovasz} with the one used in \cite{mceliece-et-al-1977} to obtain \eqref{eq:hamming}.

Observe first that Lov\'asz assignment can be written in one dimension ($n=1$) as
\[
g_1(x) = 1_0(x)+ \varphi  1_{\pm 1} (x),\quad x\in C_q\,,
\]
where  $\varphi=(2\cos(\pi/q))^{-1}$. This gives 
\begin{equation}
\widehat{g_1}(\omega)=1+2\varphi \cos(2 \pi \omega/q), \quad \omega\in C_q.
\end{equation}
Correspondingly, define the $n$-dimensional assignment 
\[
g(x) = \prod_{j=1}^n g_1(x_j), \quad 
\hat g(\omega) = \prod_{j=1}^n\widehat{g_1}(\omega_j), \quad x,\omega \in C_q^n.
\]
Note that $\widehat{g_1}\geq 0$ and, additionally, $\widehat{g_1}(\omega)=0$ for $\omega=\pm c$, with $c=(q-1)/2$. So, $\hat{g}\geq 0$, with $g(\omega)=0$ if $\omega$ contains any $\pm c$ entry.
Since $g(x)=0$ for $x\notin \{0,\pm 1\}^n$, $g$ satisfies all the properties required for $f$ in the case  $d=\infty$, and when used in \eqref{eq:f0fhat0} it gives Lov\'asz' bound
\begin{align}
|\matc| & \leq q^n\frac{g(0)}{\hat{g}(0)}\\
& = q^n \left(\frac{\cos(\pi/q)}{1+\cos(\pi/q)}\right)^n
\end{align}
for codes of infinite minimum distance.

For the case of finite $d\leq n$, we build a function $f$ of the form $f(x)=g(x)h(x)$, for an appropriate $h(x)$. In
particular, since $g(x)$ is non-negative and already takes care of setting $f(x)$ to zero if $x\notin \{0,\pm1\}^n$, it suffices to choose $h$ such that $h(x)\leq 0 $ whenever $x\in\{0,\pm1\}^n$ contains at least $d$ entries with value $\pm1$. We restrict attention to $h$ such that $\hat{h}\geq 0$, so that $\hat{f}=q^{-n}\hat{g}*\hat{h}\geq 0$. In particular, we consider functions $h$ whose Fourier transform is constant on each of the following ``spheres'' in $C_q^n$ 
\[
S_\ell^c = \{x:|\{i:x_i=\pm c \}|=\ell,\;  |\{i:x_i=0 \}|=n-\ell\}\,,\quad \ell=0,\ldots, n\,,
\]
and zero outside. This choice is motivated by the fact, observed before, that $\hat g_1(\pm c)=0$. Restricting $\hat{h}$ to be null out of these spheres simplifies the problem considerably.
We thus define
\begin{equation}\label{eq:formofh}
\hat{h}(\omega)=\sum_{\ell=0}^n \hat{h}_\ell 1_{S_{\ell}^c}(\omega)\,,\quad h(x)=q^{-n}\sum_{\ell=0}^n \hat{h}_\ell \widehat {1_{S_{\ell}^c}}(x)\,,
\end{equation}
where $\hat{h}_\ell\geq 0$ and $\hat{h}_0>0$ will be optimized later.
Since $\hat{g}(\omega)=0$, $\omega\in S_\ell\,, \ell>0$, setting $f(x)=g(x)h(x)$ gives
$\hat{f}(0)=q^{-n}(\hat{g}*\hat{h})(0)=q^{-n}\hat{g}(0)\hat{h}_0$. So, the bound \eqref{eq:f0fhat0} becomes 
\begin{equation}
|\matc|\leq \left(q^n\frac{g(0)}{\hat{g}(0)}\right)\left(q^n \frac{h(0)}{\hat{h}_0}\right)\,.
\end{equation}
The first term above is precisely Lov\'asz bound and corresponds the first term in the right hand side of  \eqref{eq:main2}. We now show that the second term corresponds to the linear programming bound of an imaginary ``Hamming scheme'' with a special non-integer alphabet size $q'=1+\cos(\pi/q)^{-1}$. To do this, define analogously to $S_\ell^c$ the spheres
 \[
S_u^1 = \{x:|\{i:x_i=\pm 1 \}|=u,\;  |\{i:x_i=0 \}|=n-u\} \,.
\]
Our constraint is that $h(x)\leq 0$ if $x\in S_u^1$, $u\geq d$. Direct computation shows that for  $x\in S_u^1$,
\begin{align}
\widehat {1_{S_{\ell}^c}}(x) & =\sum_{j=0}^\ell \binom{u}{j}\binom{n-u}{\ell-j}(-1)^j2^\ell (\cos(\pi/q))^j\,,\qquad (x\in S_u^1)\\
& = (2\cos(\pi/q))^\ell  K_\ell(u;q'),\qquad (q'=1+\cos(\pi/q)^{-1}\,),
\end{align}
where $K_\ell(u;q')$ is a Krawtchouck polynomial of degree $\ell$ and parameter $q'$ in the variable $u$.
Setting
\begin{equation}
H(u)=h(x)\,, x\in S_u^1\,,\qquad \hat{H}_\ell=q^{-n} (2\cos(\pi/q))^\ell\cdot  \hat{h}_\ell\,,
\end{equation} 
we have 
\begin{equation}
\label{eq:htoLPq'}
q^n\frac{h(0)}{\hat{h}_0} = \frac{H(0)}{\hat{H}_0}\,,
\end{equation}
where the conditions on $h$ can be restated as
\begin{align}
H(u) & =\sum_{\ell=0}^n \hat{H}_\ell K_\ell(u;q')\,,\quad u=0,\ldots,n\,,\\
\hat{H}_\ell & \geq 0 \,,\quad \ell \geq 0\,,\label{eq:hpos}\\
H(u) & \leq 0\,,\quad u\geq d\,.
\end{align}
So, the minimization of \eqref{eq:htoLPq'} is reduced to the standard linear programming problem for the Hamming space, though with a non-integer parameter $q'$. Since the construction of the polynomial used in \cite{mceliece-et-al-1977} and \cite{aaltonen1990new} can be applied verbatim for non-integer values of $q'$(see also \cite{ismail-simeonov-1998} for the position of the roots of $K_\ell(u;q')$), the claimed bound follows.
\end{proof}

\subsection{On improving bounds on the ``Plotkin'' point.}
We note that the bound~\eqref{eq:main2} matches Lovasz's~\eqref{eq:lovasz_cq} when $\delta > 1-1/q'$. Here
we show that this
threshold cannot be reduced with assignments of the type $f(x)=g(x) h(x)$ unless one chooses $h(x)$ to be a 
high degree polynomial. 
Indeed, going back to the proof of
Theorem~\ref{th:main} we notice that (from symmetry) without loss of generality we can seek $h$ among assignments of the
type
\begin{equation}\label{eq:rtt1}
	h(x) = \sum_{u=0}^n H(u) 1_{S_u^1}(x)\,,\quad \forall x\in \{0,\pm1\}^n\,,
\end{equation}
and the values at $x\not\in \{0,\pm1\}^n$ are immaterial. Here $H(u)$ is a degree $\le n$ polynomial satisfying
\begin{equation}\label{eq:rtt2a}
	H(u) \le 0 \quad \forall u\in [\delta n,n]\cap \mathbb{Z} 
\end{equation}
and a set of linear constraints, encoded in
\begin{equation}\label{eq:rtt0}
	\hat g*\hat h \ge 0\,, \quad \hat g*\hat h(0) > 0\,.
\end{equation}
For any such $H$ we get the bound (after simple manipulations)
\begin{equation}\label{eq:rtt1a}
	|\matc| \le \theta_L(C_q)^n \cdot {H(0)\over \EE[H(U)]}, \qquad U \sim \mathrm{Bino}\left(n, 1-{1\over q'}\right)\,.
\end{equation}
Next, let us (as is done customarily) impose a stronger constraint on $H$ namely that 
\begin{equation}\label{eq:rtt2}
	H(u) \le 0 \quad \forall u\in[\delta n,n]\,.
\end{equation}
We want to show that every polynomial $H$ of degree $\le 2r$ satisfying~\eqref{eq:rtt2} for $\delta < {1-{1\over q'}}$
must have $\EE[H(U)] \le 0$ for $n\gg 1$ and thus will violate~\eqref{eq:rtt0}. This is a standard result in orthogonal
polynomials, see~\cite[Lemma 2]{Sid80} for its appearance in regards to linear-programming bounds. Indeed, it is known
since Gauss that there exists a discrete measure $P_V$ with $r$ atoms located in the
roots of $K_r(\cdot; q')$ such that for any polynomial of degree $\le 2r$ we have
$$ \EE[H(U)] = \EE[H(V)]\,.$$
Since all roots of $K_r$ converge to $n\left(1-{1\over q'}\right) + o(n)$ as $n\to\infty$ we deduce from~\eqref{eq:rtt2}
that, if $\delta < 1-{1\over q'}$, $\EE[H(U)]\le 0$. Similarly, it can be shown that the MRRW choice used in
Theorem~\ref{th:main} is optimal in the sense of minimizing the degree of $H$ for a given distance, cf.~\cite{Sid80}.%

As a closing remark, notice that a choice of $H$ in Theorem~\ref{th:main} in
addition to~\eqref{eq:rtt0},~\eqref{eq:rtt2} was further restricted to positive sums of Krawtchouk polynomials,
cf.~\eqref{eq:hpos}, which is stronger than~\eqref{eq:rtt0}. However,~\cite[Theorem 3]{Sid80} shows that this extra 
restriction does not affect the bounds, unless one goes beyond the minimal-degree (MRRW) choice.
Further improvements may be more easily obtained by addressing constraint~\eqref{eq:rtt2a} 
via Elias-method (as is
done in the second linear programming bounds, cf.~\cite{aaltonen1990new}).

\subsection{ Achievability results}
\label{sec:achievability}
In this section we prove some lower bounds on $R^*(C_q,\delta)$ and in particular the left-hand side inequality~\eqref{eq:main_penta}. To that end we consider the following general
method of producing new codes from old. We note that cycles $C_q$ are in fact Cayley graphs on an abelian group, and so
the next definitions are rather natural.

Let $\Gamma$ be an abelian group with weight function $w:\Gamma \to \mreals_+\cup\{+\infty\}$ s.t. $w(x)=0 \iff x=0$. 
We define the
translation-invariant semi-metric $d(x,y)=w(x-y)$. ($d(\cdot,\cdot)$ is a metric when $w(-x)=w(x)$ and $w(a+b)\le
w(a)+w(b)$, which is the case for the Hamming weight on $\FF_q^n$, for example.) We will say that $(\Gamma_1,w_1)$ is
isomorphic to $(\Gamma_2,w_2)$ if there exists a surjective group isometry  $\Gamma_1\to\Gamma_2$.

Given a $\matc$ that is a subgroup of $\Gamma$ we form a \textit{factor group
$\Gamma'=\Gamma/\matc$} as a group of cosets of $\matc$ and define a \textit{factor-weight} on the space of cosets as
$$ w'(\gamma'+\matc) \eqdef \min_{\gamma\in\gamma'+\matc} w(\gamma)\,.$$
It is clear now that if $\matc'$ is a code in $(\Gamma', w')$ then the union of cosets $\matc'+\matc$ is a code of
minimum distance (in $(\Gamma, w)$) equal to
\begin{equation}
\label{eq:minofdmin}
d_{\min}(\matc' + \matc) = \min\{d_{\min}(\matc), d_{\min}(\matc')\}\,.
\end{equation} 

We now need the following result for the particular case of the pentagon.
\begin{proposition}
\label{prop:C'inF5}
Consider $\Gamma = \FF_5^{n}$ with weight $w(x)=\infty$ if $x_j=\pm 2 \mod 5$ for any $j\in[n]$ and
otherwise $w(x)$ is the Hamming weight of $x\in\FF_5^n$. Let $n=2k$ and take $\matc$ to be the $\FF_5$-row span of the
matrix $[1,2] \otimes I_k$. The code $\matc$ has rate ${1\over 2}\log 5$ and minimum distance $\infty$. The factor-space
$(\Gamma', w')$ is isomorophic  to the Hamming space $\FF_5^k$ with Hamming weight.
\end{proposition}
\begin{proof} The minimum distance of $\matc$ is verified easily: any non-zero $\FF_5$-multiple of $[1,2]$ has an entry
$\pm2 \mod 5$. It is a standard fact that $\matc$ is the code achieving the zero-error capacity of a pentagon.

To show the isomorphism, first consider the case $n=2,k=1$. Every vector in $\FF_5^2$ can be written in coordinates
$$ x = a [1,2] + b[0, 1]\,, \quad a,b\in\FF_5\,.$$
Then for a factor weight of such $x$ we have 
$$ w'(x+\matc) \eqdef \min_{a_1\in\FF_5} w(a_1[1,2] + b[0,1])\,.$$
When $b=\pm 1$ we take $a_1=0$ and when $b=\pm 2$ we take $a=-b$. In this way, we get
$$ w'(x+\matc) = \begin{cases} 0, x\in \matc\\
				1, x\not \in \matc\,.\end{cases} $$
Since there are $5$ cosets, $(\gamma', w')$ is isomorphic to a one-dimensional Hamming space $\FF_5$. The general case
of $k>1$ follows by vectorizing this argument.
\end{proof}

Equation \eqref{eq:minofdmin} combined with Proposition \ref{prop:C'inF5} implies the following theorem, from which we get the bound~\eqref{eq:main_penta} by application of the Gilbert-Varshamov bound for $\FF_5^k$,
see~\eqref{eq:hamming}, to construct the code $\matc'$ in the factor-space.
\begin{theorem}
\label{th:achiev_pentag}
We have
\begin{equation}
R^*(C_5,\delta)\geq \frac{1}{2}\log 5 +\frac{1}{2}R^*(K_5,2\delta)\,.
\end{equation}
\end{theorem}

By a similar procedure, we can derive achievability bounds for other cycles. For
example, when $q=2^r+1$, by using a maximal independent set in $C_q^r$ presented in \cite[Th. 3]{baumert-et-al-1971}, we get the following estimation which generalizes Theorem \ref{th:achiev_pentag}.
\begin{proposition}
\label{prop:achiev_2^r+1} 
For $q=2^r+1$, we have
\begin{equation}
R^*(C_q,\delta)\geq \frac{r-1}{r}\log q +\frac{1}{r}R^*(K_q,r\delta)\,.
\label{eq:LBgeneralr}
\end{equation}
\end{proposition}
\begin{proof} For any $n$, consider the group $\matc_q^n$ with weight $w(x)=\infty$ if $x\notin\{0,\pm 1 \}^n$ and $w(x)$ is the Hamming weight of $x$ if $x\in \{0,\pm 1\}^n$. We build a code of length $n=kr$ and show that, as $k\to \infty$, we can achieve \eqref{eq:LBgeneralr}. 
The construction is similar to the one used in Theorem \ref{th:achiev_pentag} for the particular case $r=2$, with the difference that we now only use the Hamming weight on the factor group as a lower bound for the factor-weight, since the factor groups is not isometric to a Hamming space for $r>2$.

As shown in \cite[Th. 3]{baumert-et-al-1971}, the following set
\begin{equation}
\matc_1 = \{(a_1,\ldots,a_r): a_i\in C_q, a_r=2a_1+4a_2+\cdots+2^{r-1}a_{r-1} \}
\end{equation}
is a maximal independent set of $C_q^r$. Note that by maximality the factor weight on $C_q^r/\matc_1$ is finite. Hence, we can consider an infinite distance code $\matc\subset C_q^{rk}$ with rate $(1-1/r) \log q$ as the row-``span''\footnote{Here the ``span'' is intended with coefficients in $C_q$, where multiplication is mod $q$, but $q$ need not be a prime.} of the matrix
\begin{equation}
G=\left [I_{(r-1)k},\, 
\begin{array}{c}
2 I_k\\
4I_k\\
\vdots\\
2^{r-1} I_k
\end{array}
\right]\,,
\end{equation}
where $I_m$ is a again he $m\times m$ identity matrix. The factor group $\Gamma'$ is homomorphic to $C_q^k$ and admits the representation
\begin{equation}
[x']=(0,0,\ldots,0, x') + \matc\,,\quad x'\in C_q^k\,,
\label{eq:factorelements}
\end{equation}
where there are $(r-1)k$ concatenated zeros. We see that, using this representation, any $x'\in C_q^k$ with Hamming weight $d$ gives a class $[x']$ with factor weight at least $d$. In fact, in the sum in \eqref{eq:factorelements}, due to the structure of $G$ above if a codeword in $\matc$ cancels say $t$ non-zero elements in $x'$, then it contains at least $t$ non zero components in the first $(r-1)k$ positions. Hence, any $q$-ary code of length $k$ and minimum Hamming distance $d$ gives, under the representation in \eqref{eq:factorelements}, a code $\matc'\subset \Gamma'$ with factor weight at least $d$. Setting $d=\lceil r \delta \rceil$ and using the Gilbert-Varshamov bound for codes in $C_q^k$ we get the statement of the theorem.
\end{proof}

Note that Proposition \ref{prop:achiev_2^r+1} is not uniformly better than Proposition \ref{prop:oddqweak} with the Gilbert-Varshamov bound (see Figure \ref{fig:ngon} below). However, the bound can be improved for specific $r$ using appropriate bounds for the minimum distance of codes in the factor group with the exact factor-weight. For example, in the case $r=3$ (the $9$-cycle), we have the following improved bound, also plotted in Figure \ref{fig:ngon} below.
\begin{proposition}
For $q=9$ (that is, $r=3$), we have
\begin{equation}
R^*(C_q,\delta)\geq \frac{r-1}{r}\log q +\frac{1}{r}R'(r\delta)\,.
\label{eq:LBreq3}
\end{equation}
with $R'(\delta)=0$ if $\delta>10/9$ and otherwise $R'(\delta)=\log(9)-H(P)$, where $P$ is the distribution defined by
\begin{equation}
\quad P=\frac{(1,t,t,t^2,t,t,t^2,t,t)}{1+6t+2t^2}\,,
\end{equation}
$t$ being the positive solution of the equation
\begin{equation}
2t^2(2-\delta)+6t(1-\delta)-\delta=0\,.
\end{equation}
\end{proposition}
\begin{proof}
The proof is the same as for Proposition \ref{prop:achiev_2^r+1} but we use the exact factor-weight in the Gilbert-Varshamov bound for codes in the factor group $C_q^{rk}/\matc$. In particular, for $r=3$ it is not too difficult to see that the factor weight $w'$ in $C_q^r/\matc_1$ satisfies
\begin{equation}
w'(0)=0, w'(1)=w'(2)=w'(4)=w'(5)=w'(7)=w'(8)=1, w'(3)=w'(6)=2.
\end{equation}
Since $C_q^{rk}/\matc$ is homomorphic to $C_q^k$ (consider again representation \eqref{eq:factorelements}), we can build codes in the factor group using the standard Gilbert-Varshamov procedure with the exact factor-weight. Then one deduces (see for example \cite{berlekamp-book-1984} or \cite{piret-1986}) the existence of codes with asymptotic minimum normalized distance $r\delta$ and rate
$R\geq \log(9)-H(P)$, where $P$ is the distribution with maximum entropy satisfying
\begin{equation}
\sum_{x \in C_q} P(x)w'(x)\leq r\delta\,.
\label{eq:lineqdist}
\end{equation}
Using Lagrange multipliers one deduces that the maximizing distribution is uniform if $r\delta\geq 10/9$ while for $r\delta< 10/9$ it has the form
\begin{equation}
P(x)=\frac{e^{-\lambda w'(x)}}{\sum_{x'}e^{-\lambda w'(x)}}\,,
\end{equation}
where $\lambda$ is chosen so that \eqref{eq:lineqdist} is satisfied with equality. Simple algebraic manipulations then give the claimed bound.
\end{proof}

\section{Evaluation and discussion}

Figures \ref{fig:pentagon} and \ref{fig:ngon} compare various bounds obtained for the pentagon and for the
9-gon, respectively (all logs are to the base $e$). We observe that the upper bound derived from Proposition
\ref{prop:oddqweak} with the use of a second linear programming bound  outperforms
Theorem \ref{th:main} at low values of $\delta$ (a similar result is obtained even just using the Elias bound in Proposition
\ref{prop:oddqweak}, see also \cite[Sec. V.D]{dalai-TIT-2015}) . This is not surprising since Theorem \ref{th:main} is based on the so
called first  linear programming bound of \cite{mceliece-et-al-1977}, which is known to be rather weak at high rates. 
It is an open question whether our technique (in particular working with non-integer $q$) can be extended to replace
$R_{LP1}$ in~\eqref{eq:main_penta} with $R_{LP2}$ derived by Aaltonen~\cite{aaltonen1990new} for the $q$-ary Hamming scheme.

\begin{figure}
\centering
\ifpdf
	\includegraphics[width=\linewidth]{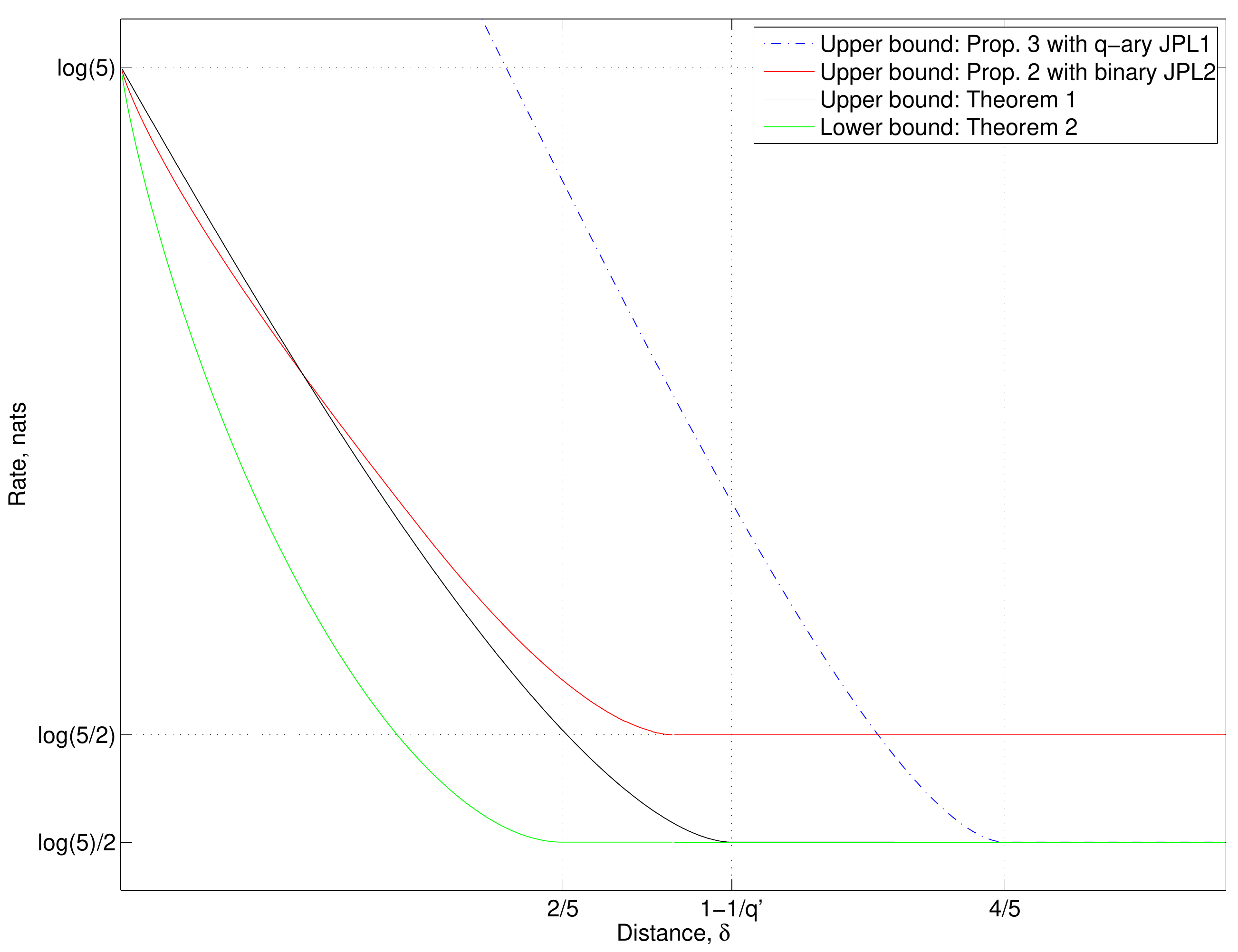}
\else
	\includegraphics[width=\linewidth]{pentagon_yp.eps}
\fi
\caption{Comparison of different bounds for the pentagon.}
\label{fig:pentagon}
\end{figure}

\begin{figure}
\centering
\ifpdf
	\includegraphics[width=\linewidth]{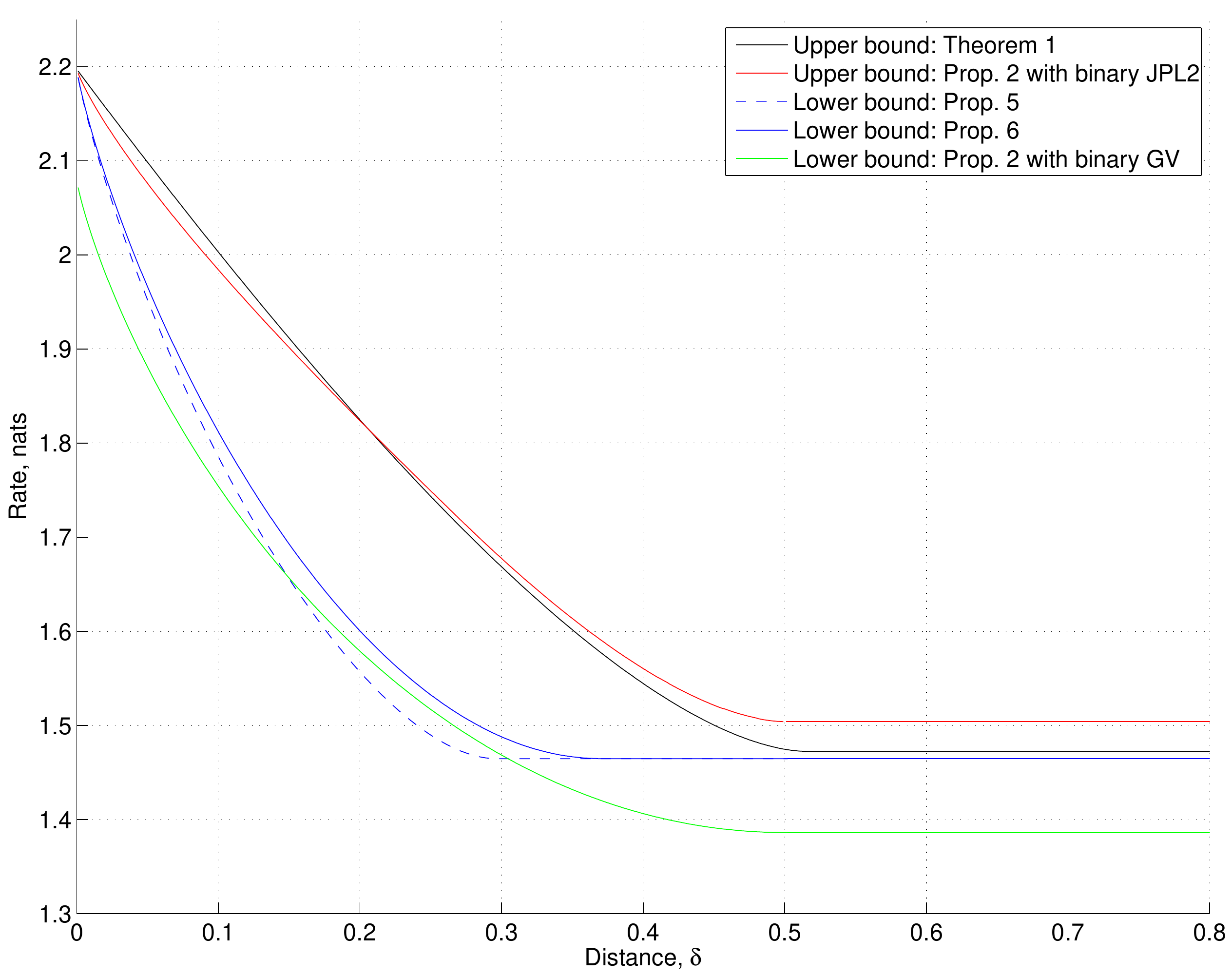}
\else
	\includegraphics[width=\linewidth]{ngon_yp.eps}
\fi
\caption{Comparison of different bounds for the 9-cycle.}
\label{fig:ngon}
\end{figure}


%

\end{document}